\documentclass[12pt]{article}
\usepackage{tikz}
\usetikzlibrary{arrows}
\usepackage[framemethod=tikz]{mdframed}
\usepackage{wrapfig}
\usepackage{amsmath,amssymb}
\usepackage{type1cm}
\usepackage{amsthm}
\usepackage{mathrsfs}
\usepackage{mathtools}
\usepackage{enumerate}
\AtBeginDocument{%
   \def\MR#1{}
}

\DeclareMathOperator{\GL}{GL}

\DeclareMathOperator{\ch}{ch}
\DeclareMathOperator{\Adm}{Adm}

\DeclareMathOperator{\cox}{cox}
\DeclareMathOperator{\ad}{ad}
\DeclareMathOperator{\Ad}{Ad}
\DeclareMathOperator{\Hom}{Hom}
\DeclareMathOperator{\id}{id}

\newcommand\F{\mathbb{F}}

\newcommand\fp{\mathfrak p}
\newcommand\vp{\varpi}
\newcommand\Fq{\mathbb F_q}
\newcommand\Fqn{\mathbb F_{q^{n_0}}}
\newcommand\p{\mathrm{pfn}}
\newcommand\aFq{\overline{\mathbb F}_q}
\newcommand\dt{\dot\tau}
\newcommand\ds{\dot s}
\newcommand\de{\dot \eta}
\newcommand\W{\mathbb W}
\newcommand\cO{\mathcal O}

\newcommand\cG{\mathcal G}
\newcommand\cF{\mathcal F}
\newcommand\sL{\mathscr L}
\newcommand\Gm{\mathbb G_m}

\newcommand\R{\mathbb R}
\newcommand\A{\mathbb A}
\newcommand\G{\mathbb G}

\newcommand\Z{\mathbb Z}
\newcommand\tW{\widetilde W}
\newcommand\SW{{^S\widetilde W}}
\newcommand\tS{\tilde S}
\newcommand\inv{\mathrm{inv}}
\newcommand\SAdm{{^S\mathrm{Adm}}}
\newcommand\lr{\lambda_{i,r}}
\newcommand\mr{\mu_{i,r}}
\newcommand\nr{\nu_{i,r}}
\newcommand\wir{w_{i,r}}
\newcommand\dwr{\dot{w}_{i,r}}
\newcommand\gbr{g_{b,i,r}}
\newcommand\gbred{g^{\mathrm{red}}_b}
\newcommand\aV{V_b^{\mathrm{adm}}}
\newcommand\aL{\mathscr{L}_{0,b}^{\mathrm{adm}}}
\newcommand\kb{\kappa(b)}
\newcommand\JO{J_{\mathcal O}}
\newcommand\Da{D_{k_0/n_0}}
\newcommand\Xr{X_{\wir}(b)}
\newcommand\XR{X_{\lr}(b)}
\newcommand\XL{X_{\wir}(b)_{\sL_0}}
\newcommand\Dr{\Omega^{n'-1}_{\Fqn}}
\newcommand\ba{\bold a}
\newcommand\CC{{of finite Coxeter type}}

\theoremstyle{definition}
\newtheorem{theo}{Theorem}[section]
\newtheorem{prop}[theo]{Proposition}
\newtheorem{defi}[theo]{Definition}
\newtheorem{lemm}[theo]{Lemma}
\newtheorem{coro}[theo]{Corollary}

\newtheorem{rema}[theo]{Remark}

\newtheorem{thm}{Theorem}[section]

\begin{document}
\title{On Some Simple Geometric Structure of Affine Deligne-Lusztig Varieties for $\GL_n$}
\author{Ryosuke Shimada}
\date{}
\maketitle

\begin{abstract}
In this paper we study the geometric structure of affine Deligne-Lusztig varieties $X_{\lambda}(b)$ for $\mathrm{GL}_n$ and $b$ basic.
Motivated by \cite{GH} and \cite{CI2}, we consider a new condition on $\lambda$.
If this is satisfied, then $X_{\lambda}(b)$ is the disjoint union of classical Deligne-Lusztig varieties times finite-dimensional affine spaces.
\end{abstract}

\section{Introduction}
Let $F$ be a non-archimedean local field with finite field $\F_q$ of prime characteristic $p$, and let $L$ be the completion of the maximal unramified extension of $F$.
Let $\sigma$ denote the Frobenius automorphism of $L/F$.
Further, we write $\cO,\ \fp$ for the valuation rings and the maximal ideal of $L$.
Finally, we denote by $\vp$ a uniformizer of $F$ (and $L$) and by $v_L$ the valuation of $L$ such that $v_L(\vp)=1$.

Let $G$ be a split connected reductive group over $\Fq$ and let $T$ be a split maximal torus of it.
Let $B$ be a Borel subgroup of $G$ containing $T$. 
For a cocharacter $\lambda\in X_*(T)$, let $\vp^{\lambda}$ be the image of $\vp\in \mathbb G_m(F)$ under the homomorphism $\lambda\colon\mathbb G_m\rightarrow T$.

We fix a dominant cocharacter $\lambda\in X_*(T)$ and $b\in G(L)$.
Then the affine Deligne-Lusztig variety $X_{\lambda}(b)$ is the locally closed reduced $\aFq$-subscheme of the affine Grassmannian defined as
$$X_{\lambda}(b)(\aFq)=\{xG(\mathcal O)\in G(L)/G(\mathcal O)\mid x^{-1}b\sigma(x)\in G(\mathcal O)\vp^{\lambda}G(\mathcal O)\}.$$
Analogously, we can also define the affine Deligne-Lusztig varieties associated to arbitrary parahoric subgroups (especially Iwahori subgroups).

The affine Deligne-Lusztig variety $X_{\lambda}(b)$ carries a natural action (by left multiplication) by the group
$$J=\{g\in G(L)\mid g^{-1}b\sigma(g)=b\}.$$
Since $b$ is usually fixed in the discussion, we mostly omit it from the notation.
We denote $J\cap G(\cO)$ by $\JO$.

The geometric properties of affine Deligne-Lusztig varieties have been studied by many people.
For example, the non-emptiness criterion and the dimension formula is already known.
Besides them, it is known that in certain cases, the (closed) affine Deligne-Lusztig variety admits a simple description.
Let $P$ be a standard parahoric subgroup of $G(L)$, where $G$ is assumed to be simple.
We define
$$X(\lambda, b)_P=\{g\in G(L)/P\mid g^{-1}b\sigma (g)\in \bigcup_{w\in\Adm(\lambda)}PwP\},$$
where $\Adm(\lambda)$ is the $\lambda$-admissible set (see \ref{Coxetertype}). 
In \cite{GH} (see also \cite{GH2}), G\"{o}rtz and He studied $X(\lambda, b)_P$ in the case that $\lambda$ is minuscule and $P$ is maximal parahoric.
They introduced a notion of ``Coxeter type'' and proved that if $(G, \lambda, P)$ is of Coxeter type and if $b$ is a basic element such that $X(\lambda, b)_P\neq \emptyset$, then $X(\lambda, b)_P$ is naturally a union of classical Deligne-Lusztig varieties.
Furthermore, the main result in \cite{GHN} determines when $X(\lambda, b)_P$ is naturally a union of classical Deligne-Lusztig varieties.
In particular, the existence of such a simple description is independent of $P$.
Finally, the work \cite{GHN3} classified the cases where $(G, \lambda, P)$ is of Coxeter type for arbitrary $P$ and $\lambda$.

In the Iwahori case, Chan and Ivanov \cite{CI2} gave an explicit description of certain Iwahori-level affine Deligne-Lusztig varieties for $\GL_n$.
Each component of the disjoint decomposition described there is a classical Deligne-Lusztig variety times finite-dimensional affine space, and they point out the similarity between their description and the results in \cite{GH}.
Using this description, they gave a geometric realization of the local Langlands correspondence.

In the hyperspecial case, the author \cite{Shimada} studied affine Deligne-Lusztig varieties for $\GL_3$ and $b$ basic.
The main theorem there determines the pairs $(\lambda, b)$ such that $X_{\lambda}(b)$ admits the simple description suggested in \cite{CI2}.

In this paper, we treat another variant of ``Coxeter case''.
Let $\tW$ be the extended affine Weyl group.
Both of the affine Weyl group $W_a$ and the finite Weyl group $W_0$ are subgroups of $\tW$.
Moreover, we have two semidirect product decomposition $\tW=W_a\rtimes \Omega$ and $\tW=W_0\ltimes X_*(T)$, where $\Omega\subset \tW$ is the subgroup of length $0$ elements.
We denote by $\tS$ the set of simple affine reflections, and by $S$ the set of simple reflections.
Then the definition of ``Coxeter type'' in \cite{GH} is concerned with the Coxeter condition on (the subgroup of) the Coxeter group $(W_a, \tS)$.
On the other hand, our definition of ``{\it finite} Coxeter type'' is based on the Coxeter condition on the {\it finite} Coxeter group $(W_0, S)$.
The purpose of this paper is to study this condition for $G=\GL_n$ and a hyperspecial parahoric subgroup $G(\cO)\subset G$.

The main results are summarized below.
If $n=3$ and $\lambda$ is not central, then these results coincide with \cite[Corollary 6.5]{Shimada}.
\begin{thm}
The dominant cocharacters $\lambda$ {\CC} are listed in Theorem \ref{class}.
\end{thm}

\begin{thm}[see Theorem \ref{maintheo}]
Let $\lambda$ be a dominant cocharacter {\CC}, and let $b$ be a basic element in $G(L)$ such that $X_{\lambda}(b)\neq\emptyset$.
Then the variety $X_{\lambda}(b)$ is a disjoint union, indexed by $J/\JO$ or $\sqcup_{j=1}^{n-1} J/\JO$, of classical Deligne-Lusztig varieties times finite-dimensional affine spaces.
Moreover, these classical Deligne-Lusztig varieties are all associated to a Coxeter element.
\end{thm}

Whenever we consider a scheme or an ind-scheme, for simplicity, we pass to the perfection even in the equal characteristic case.
Nevertheless, it is easy to check that in the equal characteristic case, all proofs and ingredients in this paper work for non-perfect rings if we speak of ``universal homeomorphism'' instead of ``isomorphism''.

The paper is organized as follows.
In Section \ref{notation} we fix notation.
In Section \ref{ADLV}, we recollect properties of affine Deligne-Lusztig varieties.
In Section \ref{condition}, we introduce a new variant of ``Coxeter case'', and give a classification using the non-emptiness criterion for affine Deligne-Lusztig varieties in the affine flag variety.
In Section \ref{structure}, we describe the geometric structure of affine Deligne-Lusztig varieties in the affine Grassmannian.
For this, we first study the Iwahori level affine Deligne-Lusztig varieties using the results in \cite{CI2} and the Deligne-Lusztig reduction method developed in \cite{GH3}.
After that, we relate the Iwahori and hyperspecial cases.

\textbf{Acknowledgments:}
I would like to thank my advisor Yoichi Mieda for his constant support and encouragement.
This work was supported by the WINGS-FMSP program at the Graduate School of Mathematical Science, the University of Tokyo. 
This work was also supported by JSPS KAKENHI Grant number JP21J22427.

\section{Notation}
\label{notation}
Throughout the paper we will use the following notation.
Let $F$ be a non-archimedean local field with finite field $\F_q$ of prime characteristic $p$, and let $L$ be the completion of the maximal unramified extension of $F$.
Let $\sigma$ denote the Frobenius automorphism of $L/F$.
Further, we write $\cO,\ \fp$ (resp.\ $\cO_F,\ \fp_F$) for the valuation rings and the maximal ideal of $L$ (resp.\ $F$).
Finally, we denote by $\vp$ a uniformizer of $F$ (and $L$) and by $v_L$ the valuation of $L$ such that $v_L(\vp)=1$.

If $F$ has positive characteristic, let $\W$ be the ring scheme over $\Fq$ where for any $\Fq$-algebra $R$, $\W(R)=R[[\vp]]$.
If $F$ has mixed characteristic, let $\W$ be the $F$-ramified Witt ring scheme over $\Fq$ so that $\W(\Fq)=\cO_F$ and $\W(\aFq)=\cO$.
Below, we restrict to the case that $R$ is a perfect $\Fq$-algebra.
In this case, the elements of $\W(R)$ can be written in the form $\Sigma_{i\geq 0}[r_i]\vp^i$, where $[r_i]$ is the Teichm\"{u}ller lift of $r_i\in R$ if $\ch F=0$ and $[r_i]=r_i$ if $\ch F>0$.
For any scheme $X$, we write $X^\p$ for the perfection of $X$.
We identify $(\W/\vp^h\W)^\p$ with the affine space $\A^{h,\p}$ under this choice of coordinates.

From now and until the end of this paper, we set $G=\GL_n$.
Let $T$ be the torus of diagonal matrices, and we choose the subgroup of upper triangular matrices $B$ as Borel subgroup. 
Further, we set $K=G(\cO)$.
Let us define the Iwahori subgroup $I\subset K$ as the inverse image of the {\it lower} triangular matrices under the projection $G(\cO)\rightarrow G(\aFq),\  \vp\mapsto 0$.

Let $\Phi=\Phi(G,T)$ denote the set of roots of $T$ in $G$.
We denote by $\Phi_+$ (resp.\ $\Phi_-$) the set of positive (resp.\ negative) roots distinguished by $B$.
Let $\chi_{ij}$ be the character $T\rightarrow \Gm$ defined by $\mathrm{diag}(t_1,t_2,\ldots, t_n)\mapsto t_i{t_j}^{-1}$.
Using this notation, we have $\Phi=\{\chi_{ij}\mid i\neq j\}$, $\Phi_+=\{\chi_{ij}\mid i< j\}$ and $\Phi_-=\{\chi_{ij}\mid i> j\}$.
We let $$X_*(T)_+=\{\lambda\in X_*(T)| \langle \alpha, \lambda \rangle\geq0\ \text{for all}\ \alpha\in \Phi_+\}$$ denote the set of dominant cocharacters.
Through the isomorphism $X_*(T)\cong \Z^n$, ${X_*(T)}_+$ can be identified with the set $\{(m_1,\cdots, m_n)\in \Z^n|m_1\geq \cdots \geq m_n\}$.

We embed $X_*(T)$ into $T(L)$ by $\lambda\mapsto \vp^{\lambda}$, where by $\vp^{\lambda}$ we denote the image of $\vp$ under the map $$\lambda\colon L^{\times}=\Gm(L)\rightarrow T(L).$$
The extended affine Weyl group $\tW$ is defined as the quotient $N_{G(L)}T(L)/T(\cO)$.
This can be identified with the semi-direct product $W_0\ltimes X_{*}(T)$, where $W_0$ is the finite Weyl group of $G$.
Let $e_1,\ldots, e_n\in L^n$ be the canonical basis.
For any permutation $\tau$ of degree $n$, we denote by $\dt$ the matrix of the form $(e_{\tau(1)}\ e_{\tau(2)} \ \cdots\ e_{\tau(n)})$.
In our case, we can identify the symmetric group of degree $n$ with $W_0$ by sending
$\tau$ to the element in $W_0$ represented by $\dt$.
We have a length function $\ell\colon \tW\rightarrow \Z_{\geq 0}$ given as
$$\ell(w_0\vp^{\lambda})=\sum_{\alpha\in \Phi_+, w_0\alpha\in \Phi_-}|\langle \alpha, \lambda\rangle+1|+\sum_{\alpha\in \Phi_+, w_0\alpha\in \Phi_+}|\langle \alpha, \lambda\rangle|,$$
where $w_0\in W_0$ and $\lambda\in X_*(T)$.

Let $S\subset W_0$ denote the subset of simple reflections, i.e., adjacent transpositions.
In this paper, let us write $s_1=(1\ 2), s_2=(2\ 3), \ldots, s_{n-1}=(n-1\ n)$.
Set $s_0=\vp^{\chi_{1,n}^{\vee}}(1\ n)$, where $\chi_{1,n}$ is the unique highest root.
The affine Weyl group $W_a$ is the subgroup of $\tW$ generated by $\tS=S\cup \{s_0\}$.
Then we can write the extended affine Weyl group as a semi-direct product $\tW\cong W_a\rtimes \Omega$, where $\Omega\subset \tW$ is the subgroup of length $0$ elements.
Moreover, $(W_a, S\cup \{s_0\})$ is a Coxeter system.
The restriction of the length function $\ell|_{W_a}$ is the length function on $W_a$ given by the fixed system of generators.
We denote by $\SW$ the set of minimal length elements for the cosets in $W_0\backslash \tW$ (cf.\ \cite[(2.4.5)]{Macdonald}).

Let $\Phi_a$ denote the set of affine roots of $T$ in $G$.
As usual we order the affine roots in such a way that the simple affine roots are the functions
$$\chi_{i,i+1}\colon \R^n\rightarrow \R,\quad (x_1,\ldots, x_n)\mapsto x_i-x_{i+1}\ (1\le i \le n-1)$$
together with the affine linear function
$$\chi_0\colon \R^n\rightarrow \R,\quad (x_1,\ldots, x_n)\mapsto x_n-x_1+1.$$
The extended affine Weyl group $\tW$ acts on $\Phi_a$.
For example, we have
$$\tau \chi_{i,i+1}=\chi_{\tau(i),\tau(i+1)},\quad\tau\in W_0$$
and
$$\vp^{\lambda}\chi_{i,i+1}=\chi_{i,i+1}-\langle\chi_{i,i+1}, \lambda\rangle \delta,\quad\lambda\in X_*(T),$$
where $\delta$ is the constant function with value $1$.
 
\section{Affine Deligne-Lusztig varieties}
\label{ADLV}
In this section we first recall the affine Grassmannian and the affine flag variety for $\GL_n$.
After that, we will recall the definition of affine Deligne-Lusztig varieties.

\subsection{The Affine Grassmannian}
Let $X$ (resp.\ $\mathcal X$) be a scheme over $F$ (resp.\ $\cO_F$).
For any perfect $\Fq$-algebra $R$, the functor $LX$ defined by $$LX(R)=X(\W(R)[\vp^{-1}])$$ is called the loop space of $X$.
The positive loop group of $\mathcal X$ is the functor defined by $$L^+\mathcal X(R)=\mathcal X(\W(R)).$$

Let $\cG$ be a smooth affine group scheme over $\cO_F$ with generic fiber $G$.
The {\it affine Grassmannian} of $\cG$ is the fpqc quotient $LG/L^+{\cG}$.
It has the following representability property.

\begin{theo}
\label{indscheme}
The fpqc-sheaf $LG/L^+{\cG}$ is represented by a strict ind-perfect scheme.
\end{theo}
\begin{proof}
This is true for any connected reductive group over $F$.
In the equal characteristic case, see \cite[Theorem 1.4]{PR}.
In the mixed characteristic case, see \cite[Corollary 9.6]{BS}.
\end{proof}

If $\cG(\cO)=K$ (resp.\ $\cG(\cO)=I$), then we call $LG/L^+{\cG}$ the {\it affine Grassmannian} (resp.\ {\it affine flag variety}) for $G$, and denote it by $\cG rass$ (resp.\ $\cF lag$).
Let $\pi$ denote the projection $\cF lag\rightarrow \cG rass$.

The affine Grassmannian for $G$ can be interpreted as parameter spaces of lattices satisfying certain conditions.
Explicitly, for any perfect $\Fq$-algebra $R$, $\cG rass(R)$ can be seen as the set of finite projective $\W(R)$-submodules $\sL$ of $\W(R)[\vp^{-1}]^n$ such that $\sL\otimes_{\W(R)}\W(R)[\vp^{-1}]=\W(R)[\vp^{-1}]^n$.
For any $\sL\in\cG rass(R)$, we define its {\it dual lattice} $\sL^*$ by $\Hom_{\W(R)}(\sL, \W(R))$, which is also a finite projective $\W(R)$-module.
Further, we have $$\sL^*\subset \sL^*\otimes_{\W(R)}\W(R)[\vp^{-1}]=\Hom_{\W(R)[\vp^{-1}]}(\W(R)[\vp^{-1}]^n, \W(R)[\vp^{-1}]).$$
Let us denote the $\W(R)[\vp^{-1}]$-module $\Hom_{\W(R)[\vp^{-1}]}(\W(R)[\vp^{-1}]^n, \W(R)[\vp^{-1}])$ by $(\W(R)[\vp^{-1}]^n)^*$.
Let $e_1,\ldots, e_n\in \W(R)[\vp^{-1}]^n$ be the canonical basis.
Let $e_1^*,\ldots, e_n^*\in (\W(R)[\vp^{-1}]^n)^*$ be its dual basis, i.e., the $\W(R)[\vp^{-1}]$-linear maps $\W(R)[\vp^{-1}]^n\rightarrow \W(R)[\vp^{-1}]$
defined as
\[
  e_i^*(e_j)=
  \begin{cases}
    1 & (i=j) \\
    0 & (i\neq j). \\
  \end{cases}
\]
Then we have an isomorphism $\W(R)[\vp^{-1}]^n\cong (\W(R)[\vp^{-1}]^n)^*$ defined by sending $e_i$ to $e_i^*$.
Under this isomorphism, we may consider $\sL^*$ as an element of $\cG rass(R)$.
\begin{prop}
\label{dual}
The morphism $*\colon \cG rass\rightarrow \cG rass$ defined by $\sL\mapsto \sL^*$ as above is an automorphism.
Moreover, the image of $g\cO^n\in \cG rass(\bar k)$ is ${^tg}^{-1}\cO^n\in \cG rass(\bar k)$.
\end{prop}
\begin{proof}
The first assertion is clear because $*^2=\id$.
For the last assertion, let $g\in G(L)$.
Then $g$ defines an automorphism $\varphi_g\colon L^n\rightarrow L^n,\  x\mapsto gx$.
Its dual $\varphi_g^*$ is represented by $^tg$ relative to the basis $e_1^*,\ldots, e_n^*$.
The assertion follows from the equation $\varphi_g^*((g\cO^n)^*)=(\cO^n)^*$.
\end{proof}

\subsection{Affine Deligne-Lusztig Varieties}
The affine Deligne-Lusztig variety is defined as follows.

\begin{defi}
The affine Deligne-Lusztig variety $X_{\lambda}(b)$ in the affine Grassmannian associated with $b\in G(L)$ and $\lambda\in X_*(T)_+$ is given by
$$X_{\lambda}(b)=\{gK\in G(L)/K\mid g^{-1}b\sigma(g)\in K\vp^{\lambda}K\}\subset \cG rass.$$
The affine Deligne-Lusztig variety $X_w(b)$ in the affine flag variety associated with $b\in G(L)$ and $w\in \tW$ is given by
$$X_w(b)=\{gI\in G(L)/I\mid g^{-1}b\sigma(g)\in IwI\}\subset \cF lag.$$
\end{defi}
Let us denote by $J_b$ the $\sigma$-centralizer of $b$.
Then the varieties $X_{\lambda}(b)$ and $X_w(b)$ are equipped with an action of $J_b(F)$.

Let $B(G)$ be the set of $\sigma$-conjugacy classes of elements in $G(L)$.
For a dominant cocharacter $\lambda$, we define a subset $B(G, \lambda)$ of $B(G)$ as the set of $b\in B(G)$ satisfying $\nu_b\preceq \lambda$, where $\nu_b$ is the Newton vector of $b$ and $\preceq$ denotes the dominance order on $X_*(T)_+$.
Inside $B(G,\lambda)$, there is always a unique basic element.
Moreover, we have the following criterion for non-emptiness of $X_{\lambda}(b)$:
\begin{theo}
The variety $X_{\lambda}(b)$ is non-empty if and only if the $\sigma$-conjugacy class of $b$ is contained in $B(G, \lambda)$.
\end{theo}
\begin{proof}
This is a combination of the results in \cite{RR} and \cite{Gashi}.
\end{proof}

For $X_w(b)$ in the affine flag variety, a complete answer to the non-emptiness question is not known.
However, if $b$ is basic, we have a criterion using the notion of ``$P$-alcove''.
To explain this, we need some notation.
Let $\ba$ be the base alcove corresponding to $I$ (as in the sense of \cite[1.2]{GHKR2}).
Let $S'\subset S$ and $w_0\in W_0$.
Let $U_{ij}$ be the root subgroup for $\chi_{ij}\in \Phi$.
We denote by $P_{S'}$ the standard parabolic subgroup corresponding to $S'$.
For any $x\in \tW$, we say $x\ba$ is a $^{w_0}P_{S'}$-alcove, if
\begin{enumerate}[(i)]
\item $w_0^{-1}x w_0\in \widetilde W_{S'}\coloneqq X_*(T)\rtimes W_{S'}$, and
\item For any $\chi_{ij}\in w_0(\Phi_+\setminus \Phi_{S'})$, $U_{ij}\cap {^xI}\subseteq U_{ij}\cap I$,
\end{enumerate}
where $W_{S'}\subset W_0$ is the subgroup generated by $S'$, and $\Phi_{S'}$ is the set of roots spanned by the simple roots corresponding to $S'$.
Finally, let $M_{S'}$ be the Levi component of $P_{S'}$, and let $\kappa_{M_{S'}}$ be the corresponding Kottwitz map.
\begin{theo}
\label{Palcove}
Let $b\in G(L)$ be a basic element, and let $x\in \tW$.
Then $X_x(b)\neq \emptyset$ if and only if, for every $(S', w_0)$ for which $x\ba$ is a $^{w_0}P_{S'}$-alcove, $b$ is $\sigma$-conjugate to an element $b'\in M_{S'}(L)$ and $x$ and $b'$ have the same image under $\kappa_{M_{S'}}$.
\end{theo}
\begin{proof}
This is conjectured in \cite{GHKR2}, and proved in \cite{GHN2}.
\end{proof}

Let $p_1$ be the projection $\tW=W_0\ltimes X_*(T)\rightarrow W_0$.
We define $p_2\colon \tW \rightarrow W_0$ as follows:
The image $p_2(w)$ is the unique element in $W_0$ such that $p_2(w)^{-1}w\ba$ is contained in the dominant chamber.
We also have a more explicit criterion for the emptiness of $X_w(b)$.
\begin{prop}
\label{emptiness}
Let $b\in G(L)$ be a basic element.
Let $w\in \tW$, and write $w=\vp^{\lambda}w_0$ with $\lambda\in X_*(T),\ w_0\in W_0$.
Assume that $\lambda\neq \nu_b$ and $p_2(w)^{-1}p_1(w)p_2(w)\in \bigcup_{S'\subsetneq S}W_{S'}$,
where $\nu_b$ is the newton vector of $b$ and $W_{S'}\subset W_0$ is the subgroup generated by $S'$.
Then $X_w(b)=\emptyset$.
\end{prop}
\begin{proof}
This is \cite[Proposition 9.5.4]{GHKR2} (see also \cite[Corollary 11.3.5]{GHKR2}).
\end{proof}

\subsection{Deligne-Lusztig Reduction Method}
For affine Deligne-Lusztig varieties $X_w(b)\subset \cF lag$, we have the following reduction method developed in \cite[2.5]{GH3} (see also \cite[Proposition 3.3.1]{HZZ}).
\begin{prop}
\label{reduction}
Let $w\in \tW$, $s\in \tS$ and $b\in G(L)$.
\begin{enumerate}[(i)]
\item Let $\eta\in \Omega$. Then there exists a $J_b(F)$-equivariant isomorphism
$$X_w(b)\xrightarrow{\sim} X_{\eta w \eta^{-1}}(b).$$
\item If $\ell(sws)=\ell(w)$, then there exists a $J_b(F)$-equivariant isomorphism
$$X_w(b)\xrightarrow{\sim} X_{sws}(b).$$
\item If $\ell(sws)=\ell(w)-2$, then $X_w(b)$ has a $J_b(F)$-stable closed subscheme $X_1$ satisfying the following conditions:
\begin{enumerate}
\item There exists a $J_b(F)$-equivariant morphism $$X_1\rightarrow X_{sws}(b)$$
which is a Zariski-locally trivial $\A^{1,\p}$-bundle.
\item Let $X_2$ be the $J_b(F)$-stable open subscheme of $X_w(b)$ complement to $X_1$.
Then there exists a $J_b(F)$-equivariant morphism $$X_2\rightarrow X_{sw}(b)$$
which is a Zariski-locally trivial $\G_m^\p$-bundle.
\end{enumerate}
\end{enumerate}
\end{prop}
\begin{proof}
All of the statements are proved in \cite[2.5]{GH3} (note that in our case, the Frobenius endmorphism $\sigma$ is an isomorphism).
\end{proof}

Here we will give the set-theoretical description of the morphisms in Proposition \ref{reduction}.
First of all, the isomorphism in (i) is given as the map 
$$gI\in X_w(b)\mapsto g\eta^{-1}I \in X_{\eta w \eta^{-1}}(b).$$
Let us denote by $\inv$ the relative position map on $\cF lag$.
If $\ell(sw)<\ell(w)$, let $C_g$ be the unique element in $\cF lag$ such that $\inv(gI, C_g)=s$ and $\inv(C_g, b\sigma(g)I)=sw$.
Using this notation, we describe the isomorphism in (ii) as the map
$$gI\in X_w(b)\mapsto C_g \in X_{sws}(b).$$
Note that in this case we always have $\inv(C_g, b\sigma (C_g))=sws$ because $\ell(sws)>\ell(sw)$.
If $\ell(sw)>\ell(w)$, then by exchanging $w$ and $sws$, we can reduce to the case $\ell(sw)<\ell(w)$.
Finally, let us explain (iii).
The set $X_1$ (resp.\ $X_2$) consists of the elements $gI\in X_w(b)$ satisfying $\inv(C_g, b\sigma(C_g))=sws$ (resp.\ $\inv(C_g, b\sigma(C_g))=sw$),
and both of the maps in (iii) are given as the map sending $gI$ to $C_g$.

\section{A Simple Condition on $\lambda$}
\label{condition}
In this section, we consider a simple condition on $\lambda\in X_*(T)_+$.
In Section \ref{structure}, we will show that if $\lambda$ satisfies this condition, then $X_{\lambda}(b)$ has a simple geometric structure.

\subsection{A decomposition of $X_{\lambda}(b)$}
\label{Coxetertype}
For any $\lambda\in X_*(T)$, the $\lambda$-admissible set $\Adm(\lambda)$ is defined as 
$$\Adm(\lambda)=\{w\in \tW\mid \text{$w\le \vp^{\lambda'}$ for some $\lambda'\in W_0\lambda$}\},$$
where $\le$ denotes the Bruhat order on $\tW$.
Set $\SAdm(\lambda)^\circ=\Adm(\lambda)\cap \SW \cap W_0\vp^{\lambda}W_0(=\SW \cap W_0\vp^{\lambda}W_0)$.
Then, by \cite[Theorem 3.2.1]{GH} (see also \cite[2.4]{GHR}), we have a decomposition
$$X_{\lambda}(b)=\bigsqcup_{w\in \SAdm(\lambda)^\circ}\pi(X_w(b))$$
for any $b\in G(L)$.
Further, let $\SAdm(\lambda)^\circ_{\cox}$ be the subset of $\SAdm(\lambda)^\circ$ defined as
$$\{w\in \SAdm(\lambda)^\circ\mid \text{$p_1(w)$ is a (twisted) Coxeter element in $W_0$}\}.$$
\begin{defi}
We say an element $\lambda\in X_*(T)_+$ is {\it \CC} if 
$$X_{\lambda}(b)=\bigsqcup_{w\in \SAdm(\lambda)^\circ_{\cox}}\pi(X_w(b)),$$
where $b$ is an element in $G(L)$ whose $\sigma$-conjugacy class is the unique basic element in $B(G, \lambda)$.
\end{defi}

\begin{rema}
Clearly, this definition can be applied to general $G$.
\end{rema}

Since affine Deligne-Lusztig varieties are isomorphic if the element $b$ is replaced by another element $b'$ in the same $\sigma$-conjugacy class, this definition does not depend on the choice of $b$.
Further, let $m\in \Z$, and let $\lambda_m=(m,\ldots, m), c_m=\vp^{\lambda_m}$.
Then $X_{\lambda}(b)$ (resp.\ $X_w(b)$) is equal to $X_{\lambda+\lambda_m}(c_m b)$ (resp.\ $X_{c_m w}(c_m b)$) as a subset of the affine Grassmannian (resp.\ the affine flag variety).
Thus, to study $X_{\lambda}(b)$, we may replace $\lambda$ by another element $\lambda'$ in $\lambda_{\ad}$, where $\lambda_{\ad}$ denotes the image of $\lambda$ in the quotient $X_*(T)/\Z\lambda_1$. 
By abuse of notation, we will write $\lambda_{\ad}=(m_1,\ldots, m_n)$ if $\lambda=(m_1,\ldots, m_n)$.

\subsection{Classification}
In this subsection, we determine $\lambda$ \CC.
For this purpose, the following observation is essential:
\begin{lemm}
\label{coxlemm}
Let $\lambda=(m_1,\ldots, m_n)\in X_*(T)_+$.
Let $w_0$ be a permutation such that $w_0(k)>w_0(l)$ implies $m_{w_0(k)}<m_{w_0(l)}$ (equivalently $m_{w_0(k)}\neq m_{w_0(l)}$) for any $k<l$.
Then $w_0\vp^{w_0^{-1}(\lambda)}\in \tW$ belongs to $\SAdm(\lambda)^\circ$.
\end{lemm}
\begin{proof}
By the assumption, we have
$$\ell(w_0\vp^{w_0^{-1}(\lambda)})=\ell(\vp^{w_0^{-1}(\lambda)})-\ell(w_0)$$
because $w_0^{-1}\lambda=(m_{w_0(1)},\ldots, m_{w_0(n)})$.
This implies both $w_0\vp^{w_0^{-1}(\lambda)}\in \Adm(\lambda)$ and $w_0\vp^{w_0^{-1}(\lambda)}\in \SW$.
\end{proof}

Using the lemma above, we will show the following proposition:
\begin{prop}
\label{NCC}
Let $\lambda=(m_1,\ldots, m_n)\in X_*(T)_+$.
Assume $n\geq 4$.
If there exists $2\le i\le n-2$ such that $m_i>m_{i+1}$ (equivalently $m_i\neq m_{i+1}$), then $\lambda$ is not {\CC}.
\end{prop}
\begin{proof}
Let us first consider the case $n\geq 5$.
In this case, it is enough to show that there exists a cycle $c=(j_1\ j_2\ \cdots\ j_n)$ such that $\ell(c)>n-1$, $c^{-1}(1)<\cdots< c^{-1}(i)$ and $c^{-1}(i+1)<\cdots< c^{-1}(n)$ for any $2\le i\le n-2$.
Indeed, by Lemma \ref{coxlemm}, $c\vp^{c^{-1}(\lambda)}$ belongs to $\SAdm(\lambda)^\circ\setminus \SAdm(\lambda)^\circ_{\cox}$ if $m_i>m_{i+1}$.
However, the variety $X_{c\vp^{c^{-1}(\lambda)}}(b)\ (b\in B(G, \lambda))$ is non-empty.
This follows from Theorem \ref{Palcove} and the fact that $w_0^{-1}c\vp^{c^{-1}(\lambda)} w_0\in \widetilde W_{S'}$ implies $S'=S$.

To find such $c$, we will use induction on $n$.
If $n=5$, $c=(1\ 3\ 5\ 2\ 4)$ (resp.\ $c=(1\ 4\ 2\ 5\ 3)$) satisfies the condition when $i=2$ (resp.\ $i=3$).
Let us suppose $n\geq 6$.
By the induction hypothesis, there exists a cycle $c=(j_1\ j_2\ \cdots\ j_{n-1})$ fixing $n$ such that $\ell(c)>n-2$, $c^{-1}(1)<\cdots< c^{-1}(i)$ and $c^{-1}(i+1)<\cdots< c^{-1}(n-1)$ for any $2\le i\le n-3$.
Set $c'=c(n-1\ n)$.
Then it is easy to check that $\ell(c')=\ell(c)+1>n-1$, $c'^{-1}(1)<\cdots< c'^{-1}(i)$ and $c'^{-1}(i+1)<\cdots< c'^{-1}(n)$ for any $2\le i\le n-3$.
Similarly, by the induction hypothesis, there exists a cycle $c=(j_1\ j_2\ \cdots\ j_{n-1})$ fixing $1$ such that $\ell(c)>n-2$, $c^{-1}(2)<\cdots< c^{-1}(n-2)$ and $c^{-1}(n-1)< c^{-1}(n)$.
Set $c''=c(1\ 2)$.
Then it is easy to check that $\ell(c'')=\ell(c)+1>n-1$, $c''^{-1}(1)<\cdots< c''^{-1}(n-2)$ and $c''^{-1}(n-1)<c''^{-1}(n)$.
This finishes the proof for the case $n\geq 5$.

We next consider the case $n=4$.
If $m_1>m_2>m_3>m_4$ or $m_1>m_2>m_3=m_4$, set $c=(1\ 3\ 2\ 4)$.
Then, by Lemma \ref{coxlemm} and Theorem \ref{Palcove}, we have $c\vp^{c^{-1}(\lambda)}\in\SAdm(\lambda)^\circ\setminus \SAdm(\lambda)^\circ_{\cox}$ and $X_{c\vp^{c^{-1}(\lambda)}}(b)\neq \emptyset$, where $b\in B(G, \lambda)$.
If $m_1=m_2>m_3>m_4$, the same is true for $c=(1\ 4\ 2\ 3)$.

If $m_1=m_2>m_3=m_4$, we may assume $\lambda=(m_1,m_1,0,0)$.
Set $w_0=(1\ 3)(2\ 4)$ and consider $w_0\vp^{w_0^{-1}(\lambda)}\in \tW$.
Then its representative $\dot w_0\vp^{w_0^{-1}(\lambda)}\in G(L)$ is a basic element with Newton vector $(\frac{m_1}{2}, \frac{m_1}{2}, \frac{m_1}{2}, \frac{m_1}{2})(\preceq \lambda)$.
Moreover, by Lemma \ref{coxlemm}, $w_0\vp^{w_0^{-1}(\lambda)}$ belongs to $\SAdm(\lambda)^\circ\setminus \SAdm(\lambda)^\circ_{\cox}$.
However, the variety $X_{w_0\vp^{w_0^{-1}(\lambda)}}(\dot w_0\vp^{w_0^{-1}(\lambda)})$ is non-empty because $\dot w_0\vp^{w_0^{-1}(\lambda)}I\in X_{w_0\vp^{w_0^{-1}(\lambda)}}(\dot w_0\vp^{w_0^{-1}(\lambda)})$.
This finishes the proof for the case $n=4$.
\end{proof}

The goal of this section is to prove the following result.
\begin{theo}
\label{class}
Let $\lambda=(m_1,\ldots, m_n)\in X_*(T)_+$, and let $\kappa$ be an integer with $0\le \kappa <n$.
Then $\lambda$ is {\CC} if and only if $\lambda_{\ad}$ has one of the following forms:
\begin{align*}
((n-1)r+\kappa, -r,\ldots, -r)&,(r,\ldots, r,-(n-1)r-\kappa), \\
((n-1)r+1+\kappa, -r,\ldots, -r,-r-1)&, (r+1, r,\ldots, r,-(n-1)r-1-\kappa),
\end{align*}
where $r\geq 1$ (resp.\ $r\geq 0$) if $\kappa=0$ (resp.\ $1\le\kappa<n$).
\end{theo}
\begin{proof}
We first prove that the condition is necessary.
Clearly, $\lambda$ is not {\CC} if $\lambda_{\ad}=(0,\ldots, 0)$ unless $n=1$.
So, by Proposition \ref{NCC}, we are reduced to treat the following three cases:
\begin{enumerate}[(i)]
\item $m_1>m_2=\cdots=m_n$,
\item $m_1=\cdots=m_{n-1}>m_n$,
\item $m_1>m_2=\cdots=m_{n-1}>m_n$.
\end{enumerate}

If $\lambda$ satisfies (i) (resp.\ (ii)) and $m_1+\cdots+m_n=\kappa$ (resp. $m_1+\cdots+m_n=-\kappa$), then it is easy to see that $\lambda$ has one of the forms in the theorem.

Let $\lambda$ be a dominant cocharacter satisfying (iii).
If, moreover, $\lambda$ is {\CC}, then we have $m_1-m_2=1$ or $m_{n-1}-m_n=1$.
To show this, let $s=(1\ n)$ and consider $s\vp^{s^{-1}(\lambda)}\in \tW$.
This belongs to $\SAdm(\lambda)^\circ\setminus \SAdm(\lambda)^\circ_{\cox}$ by Lemma \ref{coxlemm}.
So it is enough to show that if $m_1-m_2>1$ and $m_{n-1}-m_n>1$, then $X_{s\vp^{s^{-1}(\lambda)}}(b)\neq \emptyset$, where $b\in B(G, \lambda)$.
If $(w_0, S')$ is a pair satisfying $w_0^{-1}s\vp^{s^{-1}(\lambda)} w_0\in \widetilde W_{S'}$, then 
$$\{s_{\min\{w_0^{-1}(1),w_0^{-1}(n)\}}, s_{\min\{w_0^{-1}(1),w_0^{-1}(n)\}+1},\ldots, s_{\max\{w_0^{-1}(1),w_0^{-1}(n)\}-1}\}\subseteq S'.$$
If, moreover, $\chi_{1, w_0^{-1}(1)}$ and $\chi_{w_0^{-1}(n), n}$ are both contained in $\Phi_{S'}$, then $S'=S$.
Indeed, $\chi_{1, w_0^{-1}(1)}\in \Phi_{S'}$ and $\chi_{w_0^{-1}(n), n}\in \Phi_{S'}$ imply $\{s_1, s_2,\ldots, s_{w_0^{-1}(1)-1}\}\subseteq S'$ and $\{s_{w_0^{-1}(n)}, s_{w_0^{-1}(n)+1},\ldots, s_{n-1} \}\subseteq S'$, respectively.
Note that the $(w_0(1), 1)$-th (resp.\ $(n, w_0(n))$-th) entry of $^{s\vp^{s^{-1}(\lambda)}}I$ is $\fp^{m_{w_0(1)}-m_1+1}=\fp^{m_2-m_1+1}$ (resp.\ $\fp^{m_n-m_{w_0(n)}+1}=\fp^{m_n-m_{n-1}+1}$) when $S'\neq S$.
Thus, if $m_1-m_2>1$ and $m_{n-1}-m_n>1$, there is no pair $(w_0, S')$ such that $s\vp^{s^{-1}(\lambda)}\ba$ is a $^{w_0}P_{S'}$-alcove unless $S'=S$.
By Theorem \ref{Palcove}, this implies $X_{s\vp^{s^{-1}(\lambda)}}(b)\neq \emptyset$ for $b\in B(G, \lambda)$.

To finish the proof for the necessity, it remains to show that $\lambda=(1,0,\ldots,0,-1)$ is not {\CC}.
By Lemma \ref{coxlemm}, $s_0$ belongs to $\SAdm(\lambda)^\circ\setminus \SAdm(\lambda)^\circ_{\cox}$.
However, $\dot s_0=\vp^{\chi_{1,n}^{\vee}}\dot{(1\ n)}\in G(L)$ is a basic element with Newton vector $(0,\ldots, 0)(\preceq \lambda)$ and $\dot s_0I$ is contained in $X_{s_0}(\dot s_0)$, i.e., $X_{s_0}(\dot s_0)\neq \emptyset$.
This means that $\lambda$ is not {\CC}.

For the sufficiency, we have to show $X_w(b)=\emptyset$ for any $w\in \SAdm(\lambda)^\circ\setminus \SAdm(\lambda)^\circ_{\cox}$, where $\lambda_{\ad}$ is one of the cocharacters in the theorem and $b\in B(G, \lambda)$.
As explained in \S\ref{Coxetertype}, it is enough to show this for a representative in $\lambda_{\ad}$.

If $\lambda=((n-1)r+\kappa, -r,\ldots, -r)$ with $\kappa, r$ in the theorem, then it is easy to check
$$\SAdm(\lambda)^\circ=\{w_0\vp^{w_0^{-1}(\lambda)}\mid \text{$w_0=(1\ 2\ \cdots\ j)$ for some $1\le j\le n$}\}.$$
If $w_0=(1\ 2\ \cdots\ j)$, then $p_1(w_0\vp^{w_0^{-1}(\lambda)})=w_0, p_2(w_0\vp^{w_0^{-1}(\lambda)})=1$, and hence $$p_2(w_0\vp^{w_0^{-1}(\lambda)})^{-1}p_1(w_0\vp^{w_0^{-1}(\lambda)})p_2(w_0\vp^{w_0^{-1}(\lambda)})=w_0.$$
By Proposition \ref{emptiness}, $X_{w_0\vp^{w_0^{-1}(\lambda)}}(b)=\emptyset$ for $b\in B(G, \lambda)$ and $w_0=(1\ 2\ \cdots\ j)$ with $1\le j\le n-1$.
Thus $\lambda$ is {\CC} in this case.
The proof for the case $\lambda=(r,\ldots, r,-(n-1)r-\kappa)$ is similar.

If $\lambda=((n-1)r+1+\kappa, -r,\ldots, -r,-r-1)$ with $\kappa, r$ in the theorem, then any element in $\SAdm(\lambda)^\circ$ can be written as $w_0\vp^{w_0^{-1}(\lambda)}$ for some $w_0\in W_0$.
If $w_0\vp^{w_0^{-1}(\lambda)}\in\SAdm(\lambda)^\circ$ satisfies $w_0^{-1}(1)<w_0^{-1}(n)$, then we have $$w_0=(1\ 2\ \cdots\ w_0^{-1}(1))(n\ n-1\ \cdots\ w_0^{-1}(n)).$$
Using Proposition \ref{emptiness}, we can check that $X_{w_0\vp^{w_0^{-1}(\lambda)}}(b)$ is empty for $b\in B(G, \lambda)$.
If $w_0\vp^{w_0^{-1}(\lambda)}\in\SAdm(\lambda)^\circ$ satisfies $w_0^{-1}(1)>w_0^{-1}(n)$, then we have $$w_0=(1\ 2\ \cdots\ w_0^{-1}(n)\ n\ n-1\ \cdots\ w_0^{-1}(1)).$$
This is a Coxeter element if and only if $w_0^{-1}(1)=w_0^{-1}(n)+1$.
Set $k=w_0^{-1}(1),\ l=w_0^{-1}(n)$.
Let $\de$ be the matrix of the form
$\begin{pmatrix}
0 & \vp \\
1_{n-1} & 0
\end{pmatrix}$.
Then its image $\eta$ in $\tW$ is contained in $\Omega$.
Further, $\eta^{n-k+1}w_0\vp^{w_0^{-1}(\lambda)}\eta^{-(n-k+1)}$ is equal to
$$(n-k+2\ n-k+3\ \cdots \ n-k+l+1 \ n-k+1\ n-k\ \cdots\ 1)\vp^{((n-1)r+\kappa,-r,\ldots, -r)}.$$
Note that the cocharacter $((n-1)r+\kappa,-r,\ldots, -r)$ is not central by assumption on $\kappa, r$.
So, by Proposition \ref{reduction} (i) and Proposition \ref{emptiness}, we have $X_{w_0\vp^{w_0^{-1}(\lambda)}}(b)=\emptyset$ for $b\in B(G,\lambda)$, unless $k=l+1$.
Thus $\lambda$ is {\CC} in this case.
The proof for the case $\lambda=(r+1, r,\ldots, r,-(n-1)r-1-\kappa)$ is similar.
\end{proof}

\section{Geometric Structure}
\label{structure}
Set $\kappa=v_L\circ \det$.
Fix a basic element $b\in G(L)$ with $0\le \kb<n$.
Put $n'=\mathrm{gcd}(\kappa(b),n)$, and let $n_0, k_0$ be the non-negative integers such that
$$n=n'n_0,\quad\kappa(b)=n'k_0.$$
To study the geometric structure, we choose a representative $b$ of the basic $\sigma$-conjugacy class $[b]$.
In this section, we will work with the {\it special representative} $b_{sp}$ attached to $\kb$ introduced in \cite[Definition 5.2]{CI2}.
This is the block-diagonal matrix of size $n\times n$ with $(n_0\times n_0)$-blocks of the form
${\begin{pmatrix}
0 & \vp \\
1_{n_0-1} & 0\\
\end{pmatrix}}^{k_0}$.

Let $J=J_b(F)$ and let $\JO=J\cap K$ be a maximal compact subgroup of $J$.
In our case, $J$ is isomorphic to $\GL_{n'}(\Da)$, where $D_{k_0/n_0}$ denotes the central division algebra over $F$ with invariant $k_0/n_0$.
Let $\cO_{\Da}$ be the ring of integers of $\Da$.
Then $\JO$ is isomorphic to $\GL_{n'}(\cO_{\Da})$.

\subsection{The Iwahori case}
\label{Iwahori}
We keep the notation above and assume $b=b_{sp}$.
From now on we will consider the dominant cocharacters
\[
  \lr=
  \begin{cases}
    ((n-1)r+\kb, -r,\ldots,-r) & (i=0) \\
    ((n-1)r+1+\kb, -r,\ldots,-r,-r-1) & (i=1),
  \end{cases}
\]
where $r>0$ (resp.\ $r\geq 0$) if $\kb=0$ (resp.\ $1\le \kb<n$).
Let us first study the Iwahori level affine Deligne-Lusztig varieties lying over $\XR$.
These are (possibly empty) varieties of the form $X_w(b)$ with $w\in W_0\vp^{\lr}W_0$.
Among them, we especially consider the parameters defined as follows:
Let $X_*(T)_{\lr}$ be the finite set of cocharacters obtained by the permutation of coordinates of $\lr$ fixing the first entry.
Clearly we have
\[
  |X_*(T)_{\lr}|=
  \begin{cases}
    1 & (i=0) \\
    n-1 & (i=1).
  \end{cases}
\]
In particular, $X_*(T)_{\lr}$ contains an element
\[
  \lr'=
  \begin{cases}
    \lambda_{0,r} & (i=0) \\
    ((n-1)r+1+\kb, -r-1, -r, \ldots,-r) & (i=1).
  \end{cases}
\]
Let $\tau=(1\ 2\ \cdots\  n)\in W_0$.
In this paper, affine Deligne-Lusztig varieties associated with
$$\vp^{\nr}\tau\in\tW,\quad \nr\in X_*(T)_{\lr}$$
play an important role.
In the sequel, we first study the case $\nr=\lr'$ using the results in \cite[Section 6]{CI2}.
Write $\dwr=\vp^{\lr'}\dt$ and $\wir=\vp^{\lr'}\tau\in \tW$ (note that if $i=0$, then $w_{0, r}$ is the same as the parameter studied there).

Set $V=L^n$ and $\sL_0=\cO^n$.
We define
\begin{align*}
g_b(x)=(x\ \ b\sigma(x)\ \ \cdots\ \  (b\sigma)^{n-1}(x)).
\end{align*}
Then the {\it admissible subset} $\aV$ of the isocrystal $(V, b\sigma)$ consists of the elements $x\in V$ satisfying $\det g_b(x)\in L^{\times}$.
We also define 
\begin{align*}
D_b=\mathrm{diag}(1, \vp^{\lfloor k_0/n_0\rfloor}, \vp^{\lfloor 2k_0/n_0\rfloor}, \ldots, \vp^{\lfloor (n-1)k_0/n_0\rfloor}),
\end{align*}
and set $\gbred(x)=g_b(x)D_b^{-1}$.
Analogously, let us denote by $\aL$ the subset of $\sL_0$ consisting of the elements $x\in \sL_0$ satisfying $\det \gbred(x)\in \cO^{\times}$.
Set 
\[
  \mr=
  \begin{cases}
    (0, r, 2r,\ldots, (n-1)r) & (i=0) \\
    (0,r+1, 2r+1,\ldots, (n-1)r+1) & (i=1).
  \end{cases}
\]
We write $\gbr(x)=g_b(x)\vp^{\mr}$ for any $x\in \aV$.

\begin{lemm}
\label{polygon}
Let $x\in \aV$.
Then there exist unique elements $\alpha_j\in \cO$ such that $(b\sigma)^n(x)=\Sigma_{j=0}^{n-1}\alpha_j(b\sigma)^j(x)$ with $v_L(\alpha_0)=\kb$.
Moreover, if $\kb>0$, we also have $v_L(\alpha_j)>0$ for $1\le j\le n-1$.
\end{lemm}
\begin{proof}
The first assertion is \cite[Lemma 6.1]{CI2}.
If $\kb>0$, then the slope of the Newton polygon of $(V, b\sigma)$ is positive.
So the last assertion also follows from the proof of \cite[Lemma 6.1]{CI2}.
\end{proof}

\begin{lemm}
\label{gbr}
Let $x\in \aV$.
We have 
$$b\sigma(\gbr(x))=\gbr(x)\dwr a,$$
where $a\in I$ is a matrix, which can differ from the identity matrix only in the last column.
\end{lemm}
\begin{proof}
The proof follows along the same line as \cite[Lemma 6.7]{CI2}.
\end{proof}

For an integer $m$, let $0\le [m]_{n_0}<n_0$ denote its residue modulo $n_0$.
Let $v_0\in \GL_{n_0}(L)$ be the permutation matrix whose $j$-th column is $e_{1+[(j-1)k_0]_{n_0}}$.
Let $v\in \GL_n(L)$ denote the block-diagonal matrix, whose $n_0\times n_0$ blocks are each equal to $v_0$.
Further, let $\Dr$ be the perfection of $n'-1$-dimensional Drinfeld's upper half-space over $\Fqn$.
The following statements are generalizations of \cite[Theorem 6.5 (ii)]{CI2}, \cite[Proposition 6.12]{CI2}, \cite[Proposition 6.15]{CI2} and \cite[Theorem 6.17]{CI2}.

\begin{prop}
\label{ADLVwir}
\begin{enumerate}[(i)]
\item The map $$\aV\rightarrow \Xr,\quad x\mapsto \gbr(x)$$
is surjective.
\item Let $\XL$ be the image of $\aL\subset \aV$ by the map in (i).
Then we have a scheme theoretic disjoint union decomposition
$$\Xr=\bigsqcup_{h\in J/\JO}h\XL.$$
\item The variety $\XL$ in (ii) is a locally closed subvariety of the Schubert cell $IvD_b\vp^{\mr} I/I$ and is isomorphic to $$\Dr\times \A^\p.$$
Here $\A$ is a finite-dimensional affine space over $\Fq$ with dimension depending on $i, r$. 
In particular, we have a decomposition of $\Fq$-schemes
$$\Xr\cong \bigsqcup_{J/\JO}\Dr\times \A^\p.$$
\end{enumerate}
\end{prop}
\begin{proof}
The case $i=0$ and $r>0$ is proved in \cite[Section 6]{CI2}.
If $\kb>0$, then we can check that the case $r=0$ also follows from the same proof as in \cite[Section 6]{CI2}, using Lemma \ref{polygon}.

We have to show the case $i=1$.
Let $\de$ be the matrix of the form
${\begin{pmatrix}
0 & \vp \\
1_{n-1} & 0\\
\end{pmatrix}}$, and let $\eta$ be its image in $\Omega$.
Then we have $$\de^{-1} \dot w_{1,r}\de=\vp^{(-r,\ldots,-r, (n-1)r+\kb)}\dt,$$
where $\tau=(1\ \cdots \ n)$.
Let us write $A=\vp^{(-r,\ldots,-r, (n-1)r+\kb)}\dt$.
By Proposition \ref{reduction} (i), the following map is an isomorphism:
$$\varphi\colon X_{w_{1,r}}(b)\rightarrow X_A(b),\quad gI\mapsto g\de I.$$
Moreover, when we see $A$ as an element of $\tW$,
it follows that
\begin{align}
\label{length}
\ell(s_j\cdots s_{n-2} s_{n-1} A s_{n-1}s_{n-2} \cdots s_j)=\ell(s_{j+1}\cdots s_{n-2}s_{n-1} A s_{n-1}s_{n-2} \cdots s_{j+1} )-2,
\end{align}
for any $1\le j\le n-1$.
Indeed, we compute
$$\ell(A)=\ell(w_{1,r})=(n-1)(nr+1+\kb)$$
and
$$\ell(s_1\cdots s_{n-1}As_{n-1}\cdots s_1)=\ell(w_{0,r})=(n-1)(nr-1+\kb).$$
Since $\ell(A)-\ell(s_1\cdots s_{n-1}As_{n-1}\cdots s_1)=2(n-1)$, we must have (\ref{length}) for all $1\le j\le n-1$.

Next, we claim 
\begin{align}
\label{empty}
\text{$X_{s_{j}\cdots s_{n-2}s_{n-1}As_{n-1}s_{n-2} \cdots s_{j+1}}(b)=\emptyset$ for all $1\le j \le n-1$.}
\end{align}
By Lemma \ref{emptylemm} and Proposition \ref{reduction} (ii), this is equivalent to
\begin{align}
\label{empty'}
\text{$X_{s_{j}\cdots s_2s_1A's_1s_2 \cdots s_{j-1}}(b)=\emptyset$ for all $1\le j \le n-1$,} \tag{$\ast$}
\end{align}
where $A'=\dt\vp^{\lambda_{0,r}}$.
To check ($\ast$), we use Proposition \ref{emptiness}.
Set $$A_j=s_{j}\cdots s_2s_1A's_1s_2 \cdots s_{j-1}.$$
Then we have to show 
$$p_2(A_j)^{-1}p_1(A_j)p_2(A_j)\in \bigcup_{S'\subsetneq S}W_{S'},$$
where $W_{S'}\subset W_0$ is the subgroup generated by $S'$.
For $1\le j\le n-1$, we compute
$$p_1(A_j)=s_1\cdots s_{j-1} s_{j+1}\cdots s_{n-1} ,\quad p_2(A_j)=s_{j+1}\cdots s_{n-1}.$$
So we have $$p_2(A_j)^{-1}p_1(A_j)p_2(A_j)=s_1\cdots s_{j-1} s_{j+1}\cdots s_{n-1}\in W_{S_j}\subset \bigcup_{S'\subsetneq S}W_{S'},$$
where $S_j=S\setminus \{s_j\}$.
This proves (\ref{empty'}) and hence (\ref{empty}).

Combining (\ref{length}) and (\ref{empty}) and using Proposition \ref{reduction} (iii), we can deduce that there exists a Zariski-locally trivial $\A^{1,\p}$-bundle
$$\pi_{j}\colon X_{s_{j+1}\cdots s_{n-1} A s_{n-1} \cdots s_{j+1}}(b)\rightarrow X_{s_{j}\cdots s_{n-1} A s_{n-1} \cdots s_{j}}(b),$$
for each $1\le j\le n-1$.
Moreover, for any fixed $g_0I\in X_{s_{j}\cdots s_{n-1} A s_{n-1} \cdots s_{j}}(b)$, we have 
$$\pi_{j}^{-1}(g_0I)=\{gI\in \cF lag\mid \inv(gI, g_0I)=\inv(g_0I, gI)=s_j\}.$$
Using these morphisms, we prove the case $i=1$.
For simplicitiy, we treat the case that $\kb=0$ or $\kb=1$, so that $v=D_b=1$.
The general case follows in the same way.
By the case $i=0$, $X_{w_{0,r}}(b)_{\sL_0}$ is contained in $I\vp^{\mu_{0,r}}I/I$.
Moreover, the proof of \cite[Theorem 6.17]{CI2} shows that $gI\in \cF lag$ lies in $X_{w_{0,r}}(b)_{\sL_0}$ if and only if $gI$ is represented by the element of the form
$$\begin{pmatrix}
1 &0 & \cdots & \cdots & \cdots & 0 \\
x_2 & 1 & 0 & \cdots & \cdots & 0 \\
x_3 & \ast & 1 & \cdots &\cdots & 0 \\
\vdots & \ddots & \ddots & \ddots & \vdots & \vdots \\
x_{n-1} & \ast & \cdots & \ast & 1 & 0 \\
x_n & \ast & \cdots & \ast & \ast & 1
\end{pmatrix}\vp^{\mu_{0,r}}I,$$
where
\begin{align*}
x_j=[x_{j,0}]+[x_{j,1}]\vp+\cdots+[x_{j, (j-1)r-1}]\vp^{(j-1)r-1}, \\
(x_{2,0}, \ldots ,x_{n,0})\in \Dr,\quad x_{j,k}\in \A^{1, \p}\ (k>0),
\end{align*}
and the entries marked by $\ast$ are certain functions of $x_2, \ldots, x_n$ lying in $\cO$.
The coefficient of $\vp^l$ in each $\ast$ is actually a function of only $x_{j,k}$ with $k\le l$.
Then we easily verify that $gI\in X_{w_{1, r}}(b)$ lies in the inverse image of $X_{w_{0,r}}(b)_{\sL_0}$ under the morphism 
$$\pi=\pi_1\circ \cdots\circ \pi_{n-1}\circ \varphi\colon X_{w_{1, r}}(b)\rightarrow X_{w_{0, r}}(b)$$
if and only if $gI$ is represented by the element of the form
$$\begin{pmatrix}
1 &0 & \cdots & \cdots & \cdots & 0 \\
x_2+[t_2]\vp^r & 1 & 0 & \cdots & \cdots & 0 \\
x_3+[t_3]\vp^{2r} & \ast & 1 & \cdots &\cdots & 0 \\
\vdots & \ddots & \ddots & \ddots & \vdots & \vdots \\
x_{n-1}+[t_{n-1}]\vp^{(n-2)r} & \ast & \cdots & \ast & 1 & 0 \\
x_n+[t_n]\vp^{(n-1)r} & \ast & \cdots & \ast & \ast & 1
\end{pmatrix}\vp^{\mu_{0,r}}\ds_1\cdots \ds_{n-1}\eta^{-1}I,$$
where $x_j$ are as above, $t_j\in \A^{1, \p}$, and $\ast$ are certain functions of $x_2, \ldots, x_n$, $t_2, \ldots, t_n$ lying in $\cO$.
This is true even for the case $r=0$ and $\kb>0$, but we need a little more attention.
So, in every case, we have a decomposition
$$X_{w_{1,r}}(b)=\bigsqcup_{h\in J/\JO}\pi^{-1}(hX_{w_{0,r}}(b)_{\sL_0}),$$
and each component is a locally closed subvariety of
$$I\vp^{\mu_{0,r}}Is_1Is_2I\cdots Is_{n-1}I\eta^{-1}I/I=I\vp^{\mu_{0,r}}s_1s_2\cdots s_{n-1}\eta^{-1}I/I,$$
which is isomorphic to $\Dr\times \A^\p$.
For any $x_j$ and $t_j$ as above, set
$$x'={^t(1, x_2+[t_2]\vp^r, \ldots, x_n+[t_n]\vp^{(n-1)r})}.$$
Then, by Lemma \ref{gbr} and an argument similar to that of \cite[Theorem 6.17]{CI2},
it follows that $gI\in\cF lag$ lies in $\pi^{-1}(\vp X_{w_{0,r}}(b)_{\sL_0})$ if and only if $gI=g_{b,1,r}(x')I$ for some $x'$ as above.
This implies (i).
Since we have $$X_{w_{1,r}}(b)_{\sL_0}=\pi^{-1}(\vp X_{w_{0,r}}(b)_{\sL_0})\cong \Dr\times \A^\p,$$
(ii) and (iii) also follow.
\end{proof}

For $w, w'\in \tW$, we write $w\rightarrow w'$ if there is a sequence $w=w_0, w_1,\ldots, w_m=w'$ of elements in $\tW$ such that for any $k$, $w_k=s_{j_k}w_{k-1}s_{j_k}\  (s_{j_k}\in S)$ and $\ell(w_k)\le \ell(s_{j_k}w_{k-1}s_{j_k})$.
We write $w\approx w'$ if $w\rightarrow w'$ and $w'\rightarrow w$.
\begin{lemm}
\label{emptylemm}
We have 
\begin{align*}
 w_{0,r}s_1&\approx w_{0,r}s_{n-1}, \\
 s_1w_{0,r}s_1s_2&\approx s_{n-1}w_{0,r}s_{n-1}s_{n-2}, \\
 &\ \:  \vdots \\
 s_{n-2}\cdots s_1 w_{0,r}s_1\cdots s_{n-1}&\approx s_2\cdots s_{n-1}w_{0,r}s_{n-1}\cdots s_1.
\end{align*}
\end{lemm}
\begin{proof}
Let us first show
\begin{align*}
 w_{0,r}s_1&\approx w_{0,r}s_1, \\
 s_1w_{0,r}s_1s_2&\approx w_{0,r}s_1s_2s_1, \\
 &\ \:  \vdots \\
 s_{n-2}\cdots s_1 w_{0,r}s_1\cdots s_{n-1}&\approx w_{0,r}s_1\cdots s_{n-1}s_{n-2}\cdots s_1.
\end{align*}
We have $w_{0,r}s_1\cdots s_js_{j-1}\cdots s_1=\vp^{\lambda_{0,r}}(1\ j+2\ \cdots\ n)(2\ \cdots\ j+1)$ 
and
\begin{align*}
\ell(w_{0,r}s_1\cdots s_js_{j-1}\cdots s_1)&=(n-1)(nr-1+\kb)+2j-1 \\
&=\ell(s_{j-1}\cdots s_1w_{0,r}s_1\cdots s_j).
\end{align*}
For any $1\le k\le j$, set $v_{j,k}=s_{j-k}\cdots s_2 s_1w_{0,r}s_1\cdots s_j s_{j-1}s_{j-2}\cdots s_{j-k+1}$. 
Then, for any $1\le k\le j-1$,  we also have
\[
  v_{j,k}\chi_{j-k, j-k+1}=
  \begin{cases}
   \chi_{j-k+2, j+2} & (1\le j\le n-2) \\
   \chi_{j-k+2, j+2}+(nr+\kb)\delta & (j=n-1),
  \end{cases}
\]
where $\delta$ is the constant function with value $1$.
In particular, for fixed $1\le j\le n-1$, $v_{j,k}\chi_{j-k, j-k+1}$ is always positive.
This implies that $v_{j,1}$ can be transformed to $v_{j, j}$ with $\ell(v_{j,1})\le \ell(v_{j,2})\le \cdots \le \ell({v_{j,j}})$.
Since $\ell(v_{j,1})=\ell(v_{j,j})$, it follows that $v_{j,1}\approx v_{j,j}$ for any $1\le j\le n-1$, as we claimed.
In the same way we can show
\begin{align*}
 w_{0,r}s_{n-1}&\approx s_{n-1}w_{0,r}, \\
 s_{n-1}w_{0,r}s_{n-1}s_{n-2}&\approx s_{n-1}s_{n-2}s_{n-1}w_{0,r}, \\
 &\ \:  \vdots \\
 s_{2}\cdots s_{n-1} w_{0,r}s_{n-1}\cdots s_1&\approx s_{n-1}\cdots s_{1}s_{2}\cdots s_{n-1}w_{0,r}.
\end{align*}

By the discussion above, our statement is reduced to the equivalence
$$w_{0,r}s_1\cdots s_js_{j-1}\cdots s_1\approx s_{n-1}\cdots s_{n-j}s_{n-j+1}\cdots s_{n-1}w_{0,r}$$
for all $1\le j\le n-1$.
If $j=n-1$, we have to show the equivalence between 
$$w_{0,r}s_1\cdots s_{n-1}s_{n-2}\cdots s_1=\vp^{\lambda_{0,r}}(2\ \cdots\ n)$$
and
$$s_{n-1}\cdots s_{1}s_{2}\cdots s_{n-1}w_{0,r}=(1\ n)\vp^{\lambda_{0,r}}(1\ \cdots\ n).$$
In this case, it is easy to check that the transformation (by the conjugation by an element of $S$)
\begin{align*}
\vp^{\lambda_{0,r}}(2\ \cdots\ n)&\rightarrow s_1\vp^{\lambda_{0,r}}(2\ \cdots\ n)s_1 \\
&\rightarrow \cdots \\
&\rightarrow s_{n-1}\cdots s_1\vp^{\lambda_{0,r}}(2\ \cdots\ n)s_1\cdots s_{n-1} \\
&=(1\ n)\vp^{\lambda_{0,r}}(1\ \cdots\ n)
\end{align*}
gives this equivalence.

If $1\le j\le n-2$, we compute 
$$w_{0,r}s_1\cdots s_js_{j-1}\cdots s_1=\vp^{\lambda_{0,r}}(1\ j+2\ \cdots\ n)(2\ \cdots\ j+1)$$
and
$$s_{n-1}\cdots s_{n-j}s_{n-j+1}\cdots s_{n-1}w_{0,r}=\vp^{\lambda_{0,r}}(1\ \cdots\  n-j-1\ n)(n-j\ \cdots\  n-1).$$
Here we used $\vp^{\lambda_{0,r}}s_k=s_k\vp^{\lambda_{0,r}}\ (2\le k\le n-1)$. 
Both of these elements have length $(n-1)(nr-1+\kb)+2j-1$.
We are going to show that the transformation
\begin{align*}
&\vp^{\lambda_{0,r}}(1\ j+2\ \cdots\ n)(2\ \cdots\ j+1) \\
\rightarrow &\vp^{\lambda_{0,r}}(1\ j+1\ j+3\cdots\ n)(2\ \cdots\ j\ j+2) \\
\rightarrow &\vp^{\lambda_{0,r}}(1\ j\ j+3\cdots\ n)(2\ \cdots\ j-1\ j+1\ j+2) \\
\rightarrow &\cdots \\
\rightarrow &\vp^{\lambda_{0,r}}(1\ 2\ j+3\ \cdots\ n)(3\ \cdots\ j+2) \\
\rightarrow &\vp^{\lambda_{0,r}}(1\ 2\ j+2\ j+4 \cdots\ n)(3\ \cdots\ j+1\ j+3) \\
\rightarrow &\cdots \\
\rightarrow &\vp^{\lambda_{0,r}}(1\ \cdots\  n-j-1\ n)(n-j\ \cdots\  n-1)
\end{align*}
gives the desired equivalence.
This transformation changes $(2\ \cdots\ j+1)$ to $(n-j\ \cdots\  n-1)$ using only $s_2,\ldots s_{n-2}$.
The $W_0$-part of the element appearing in each process except for the last one is of the form
$$x_{j,k}=(1\ \cdots\ k-1\ k+j\ \cdots\ n)(k\ \cdots\ k+j-1)$$
or
$$y_{j,k,j'}=(1\ \cdots\ k-1\ k+j'\ k+j+1\ \cdots\ n)(k\ \cdots\ k+j'-1\ k+j'+1\ \cdots\  k+j),$$
where $2\le k\le n-j-1$ and $1\le j'\le j-1$.
For any $w_0\in W_0$, let $\Phi(w_0)=\{\chi\in \Phi_+\mid w_0\chi\in \Phi_-\}$.
Then we can check that
\begin{align*}
\Phi(x_{j,k})=\{\chi_{k,k+j-1},\ldots, \chi_{k+j-2,k+j-1}, \chi_{k-1,k},\ldots,&\chi_{k-1,k+j-1}, 
\chi_{1,n},\ldots, \chi_{n-1,n}\}
\end{align*}
and
\begin{align*}
\Phi(y_{j,k, j'})=\{\chi_{k-1,k+j},\ldots, &\chi_{k+j-1, k+j}, \chi_{k-1, k},\ldots \chi_{k-1, k+j'-2} \\
&\chi_{k+j', k+j'+1},\ldots,\chi_{k+j', k+j-1}, \chi_{1, n},\ldots, \chi_{n-1, n}\}.
\end{align*}
In particular, both of the sets $\Phi(x_{j,k})$ and $\Phi(y_{j,k,j'})$ contain $\{\chi_{1,n}, \ldots, \chi_{n-1, n}\}$ and $|\Phi(x_{j,k})\setminus \{\chi_{1,n}, \ldots, \chi_{n-1, n}\}|=|\Phi(y_{j,k,j'})\setminus \{\chi_{1,n}, \ldots, \chi_{n-1, n}\}|=2j-1$.
Therefore the length of each element appearing in the transformation above is always $(n-1)(nr-1+\kb)+2j-1$.
This finishes the proof.
\end{proof}

Using Proposition \ref{ADLVwir}, we obtain the following result.
\begin{coro}
\label{ADLVIwahori}
Let $\tau=(1\ 2\ \cdots\ n)$.
Then, for any $\nr\in X_*(T)_{\lr}$, there exists an irreducible component $X_{\vp^{\nr}\tau}(b)_0$ of $X_{\vp^{\nr}\tau}(b)$, which is a locally closed subvariety of a Schubert cell and is isomorphic to $\Dr\times \A^\p$.
Here $\A$ is a finite-dimensional affine space over $\F_q$.
Moreover, we have a scheme theoretic disjoint union decomposition
$$X_{\vp^{\nr}\tau}(b)=\bigsqcup_{h\in J/\JO}hX_{\vp^{\nr}\tau}(b)_0\cong \bigsqcup_{J/\JO}\Dr\times \A^\p.$$
\end{coro}
\begin{proof}
We have to show this only in the case $i=1$.
Let $\nu_j$ be the element of $X_*(T)_{\lambda_{1, r}}$ whose $(j+1)$-th entry is $-r-1$.
Then we have $\vp^{\nu_1}\tau=w_{1, r}$ and the assertion for $X_{\vp^{\nu_1}\tau}(b)$ follows from Proposition \ref{ADLVwir} by setting $X_{\vp^{\nu_1}\tau}(b)_0=X_{w_{1, r}}(b)_{\sL_0}$.
Furthermore, for any $1\le j\le n-1$, we have an isomorphism
$$\phi_j\colon X_{\vp^{\nu_j}\tau}(b)\xrightarrow{\sim} X_{\vp^{\nu_{j+1}}\tau}(b).$$
Indeed, let 
$\de={\begin{pmatrix}
0 & \vp \\
1_{n-1} & 0\\
\end{pmatrix}}$.
Then it is easy to check that 
$$(\ds_j\cdots \ds_1\de \ds_{n-1}\cdots \ds_{j+1})\vp^{\nu_j}\dt (\ds_j\cdots \ds_1\de \ds_{n-1}\cdots \ds_{j+1})^{-1}=\vp^{\nu_{j+1}}\dt$$
and
\begin{align*}
\ell(\vp^{\nu_j}\tau)&=\ell(s_{j+1}\vp^{\nu_j}\tau s_{j+1})\\
&=\cdots \\
&=\ell(s_{n-1}\cdots s_{j+1}\vp^{\nu_j}\tau s_{j+1}\cdots s_{n-1}) \\
&=\ell(\eta s_{n-1}\cdots s_{j+1}\vp^{\nu_j}\tau s_{j+1}\cdots s_{n-1}\eta^{-1}) \\
&=\ell(s_1 \eta s_{n-1}\cdots s_{j+1}\vp^{\nu_j}\tau s_{j+1}\cdots s_{n-1}\eta^{-1}s_1) \\
&=\cdots \\
&=\ell(s_j\cdots s_1 \eta s_{n-1}\cdots s_{j+1}\vp^{\nu_j}\tau s_{j+1}\cdots s_{n-1}\eta^{-1}s_1\cdots s_j)\\
&=\ell(\vp^{\nu_{j+1}}\tau). 
\end{align*}
So, by Proposition \ref{reduction} (i) and (ii), we can construct $\phi_j$ for any $j$.

Let $X_{\vp^{\nu_j}\tau}(b)_0$ be the image of $X_{\vp^{\nu_1}\tau}(b)_0$ under the isomorphism $\phi_{j-1}\circ \cdots \circ \phi_1$.
Since our assertion is true for $X_{\vp^{\nu_1}\tau}(b)$, the same assertion for $X_{\vp^{\nu_j}\tau}(b)$ follows immediately except that $X_{\vp^{\nu_j}\tau}(b)_0$ is contained in a Schubert cell.
Let $$\mu_j=\mu_{1, r}-(0, \overbrace{1, \ldots, 1}^{j-1}, 0, \ldots, 0).$$
Then, again by Proposition \ref{ADLVwir} (i), any element in $X_{\vp^{\nu_1}\tau}(b)_0$ can be written as $g_b(x)\vp^{\mu_1}I$ for some $x\in \aL$.
Further, using the set-theoretical description right after Proposition \ref{reduction}, we can easily verify that 
$$(\phi_{j-1}\circ \cdots \circ \phi_1)(g_b(x)\vp^{\mu_1}I)=g_b(x)\vp^{\mu_j}I,$$
i.e., any element in $X_{\vp^{\nu_j}\tau}(b)_0$ can be written as $g_b(x)\vp^{\mu_j}I$ for some $x\in \aL$.
By this and the same argument as in \cite[Proposition 6.15]{CI2}, we can show that $X_{\vp^{\nu_j}\tau}(b)_0$ is contained in $IvD_b\vp^{\mu_j}I/I$.
This finishes the proof.
\end{proof}
 
\subsection{The Hyperspecial Case}
We keep the notation and assumptions of \S\ref{Iwahori}.
Next we deduce the geometric structure of the hyperspecial level affine Deligne-Lusztig varieties $X_{\lr}(b)$.
To complete this, we relate the Iwahori and hyperspecial cases.
\begin{lemm}
\label{inj}
Let $\tau=(1\ 2\ \cdots \ n)$.
\begin{enumerate}[(i)]
\item The projection map $$X_{w_{0,r}}(b)\rightarrow X_{\lambda_{0,r}}(b),\quad gI\mapsto gK$$ is injective.
\item The projection map $$X_{\vp^\nu\tau}(b)\rightarrow X_{\lambda_{1, r}}(b),\quad gI\mapsto gK$$
is injective for any $\nu\in X_*(T)_{\lambda_{1, r}}$.
\end{enumerate}
\end{lemm}
\begin{proof}
To show (i), recall that any element in $X_{w_{0,r}}(b)$ is of the form $g_b(x)\vp^{\mu_{0,r}}I$ with $x\in \aV$.
So it suffices to show that for any $x,y\in \aV$, $g_b(x)\vp^{\mu_{0,r}}K=g_b(y)\vp^{\mu_{0,r}}K$ implies $g_b(x)\vp^{\mu_{0,r}}I=g_b(y)\vp^{\mu_{0,r}}I$.
If $g_b(y)\vp^{\mu_{0,r}}=g_b(x)\vp^{\mu_{0,r}}p$ with $p=(p_{ij})_{i,j}\in K$, then $$y=p_{1,1}x+p_{2,1}\vp^r b\sigma(x)+\cdots+p_{n,1}\vp^{(n-1)r}(b\sigma)^{n-1}(x).$$
By multiplying this equation by $b\sigma$, and by Lemma \ref{polygon},
we can represent each column of $g_b(y)\vp^{\mu_{0,r}}$ as a linear combination of $x, \vp^rb\sigma(x),\ldots, \vp^{(n-1)r}(b\sigma)^{n-1}(x)$, 
and their coefficients are nothing but $p_{ij}$.
This calculation shows $p_{ij}\in \fp\ (i<j)$, i.e., $p\in I$.
Therefore (i) follows.
The proof for (ii) is similar.
\end{proof}

For any $w\in \SW$, set $S_w=\max\{S'\subseteq S\mid \Ad(w)(S')=S'\}$.
\begin{prop}
\label{minimal}
For any $w\in \tW$, there exist $w'\in \SW$ and $v\in W_{S_{w'}}$ such that $w\rightarrow vw'$, where $W_{S_{w'}}\subseteq W_0$ is the subgroup generated by $S_{w'}$.
\end{prop}
\begin{proof}
This is \cite[Proposition 3.1.1]{GH}.
See also \cite[Theorem 2.5]{HN2}.
\end{proof}

We use the notation of the proof of Corollary \ref{ADLVIwahori}.
\begin{lemm}
\label{closed}
For any $j$, the image of $X_{\vp^{\nu_j}\tau}(b)$ in $X_{\lambda_{1, r}}(b)$ is closed.
\end{lemm}
\begin{proof}
First note that the projection $\cF lag\rightarrow \cG rass$ is an inductive limit of projective morphisms.
In particular, it is universally closed.
So the projection $$\bigcup_{w\in W_0\vp^{\lambda_{1,r}} W_0} X_{w}(b)\rightarrow X_{\lambda_{1,r}}(b),\quad gI\mapsto gK$$ is also a closed map. 
By this fact, it suffices to show that each $X_{\vp^{\nu_j}\tau}(b)$ is closed in $\bigcup_{w\in W_0\vp^{\lambda_{1,r}}W_0} X_{w}(b)$, i.e.,
$$\bigcup_{\substack{w\in W_0\vp^{\lambda_{1,r}} W_0\\ w\le \vp^{\nu_j}\tau}}X_w(b)=X_{\vp^{\nu_j}\tau}(b).$$
Here $\le$ denotes the Bruhat order.
By Proposition \ref{reduction} and Proposition \ref{minimal}, we are reduced to show $X_{vw}(b)=\emptyset$ for $w\in \SW \cap W_0\vp^{\lambda_{1,r}} W_0=\SAdm(\lambda_{1,r})^\circ$ and $v\in W_{S_w}$ such that $\ell(vw)<\ell(\vp^{\nu_j}\tau)$.
By the proof of Theorem \ref{class}, any element in $\SW \cap W_0\vp^{\lambda_{1,r}} W_0$ can be written as $w_0\vp^{w_0^{-1}(\lambda_{1,r})}$, where
$$w_0=(1\ 2\ \cdots\ k)(n\ n-1\ \cdots\ l),\quad k<l,$$
or
$$w_0=(1\ 2\ \cdots\ l\ n\ n-1\ \cdots\ k),\quad k>l.$$
In the first (resp.\ second) case, $S_{w_0\vp^{w_0^{-1}(\lambda_{1,r})}}$ is equal to
$$\text{$\{s_{k+1},\ldots, s_{l-2}\}$ (resp.\ $\{s_{l+1},\ldots, s_{k-2}\}$).}$$
Similarly as in the proof of Theorem \ref{class}, it is easy to check $X_{vw}(b)=\emptyset$ using $\ell(vw)<\ell(\vp^{\nu_j}\tau)$.
\end{proof}

Set
\[
b^*={^tb}^{-1},\quad \lr^*=
  \begin{cases}
    (r,\ldots,r, -(n-1)r-\kb) & (i=0) \\
    (r+1, r,\ldots, r,-(n-1)r-1-\kb) & (i=1),
  \end{cases}
\]
where $r>0$ (resp.\ $r\geq 0$) if $\kb=0$ (resp.\ $1\le \kb<n$).
Then $b^*$ is a basic element in $G(L)$ with $\kappa(b^*)=-\kappa(b)$.
Let $J^*=J_{b^*}(F)$ and let $\JO^*=J^*\cap K$ be a maximal compact subgroup of $J^*$.
\begin{theo}
\label{maintheo}
Let $\A$ be a finite-dimensional affine space over $\F_q$ with dimension depending on $i, r$.
\begin{enumerate}[(i)]
\item We have a decomposition of $\Fq$-schemes $$X_{\lambda_{0,r}}(b)\cong \bigsqcup_{J/\JO}\Dr\times \A^\p.$$
\item We have a decomposition of $\Fq$-schemes $$X_{\lambda_{1,r}}(b)\cong \bigsqcup_{j=1}^{n-1} \bigsqcup_{J/\JO}\Dr\times \A^\p.$$
\item We have a decomposition of $\Fq$-schemes $$X_{\lambda_{0,r}^*}(b^*)\cong \bigsqcup_{J^*/\JO^*}\Dr\times \A^\p.$$
\item We have a decomposition of $\Fq$-schemes $$X_{\lambda_{1,r}^*}(b^*)\cong \bigsqcup_{j=1}^{n-1} \bigsqcup_{J^*/\JO^*}\Dr\times \A^\p.$$
\end{enumerate}
\end{theo}
\begin{proof}
First we prove (i).
Since the map
$$IvD_b\vp^{\mu_{0,r}}I/I\rightarrow KD_b\vp^{\mu_{0,r}}K/K,\quad gI\mapsto gK$$
is an immersion (see, for example, \cite[Lemme 2.2]{NP}), 
it follows from Corollary \ref{ADLVIwahori} that the map $X_{w_{0,r}}(b)_{\sL_0}\rightarrow X_{\lambda_{0,r}}(b),\  gI\mapsto gK$ is also an immersion.
So we regard each $hX_{w_{0,r}}(b)_{\sL_0}(h\in J/\JO)$ as a locally closed subvariety of $X_{\lambda_{0,r}}(b)$.
Let $C_{b,r}=\frac{n(n-1)}{2}r+\sum_{j=1}^{n-1}\lfloor \frac{jk_0}{n_0}\rfloor$.
By the proof of Theorem \ref{class}, any element in $X_{\lambda_{0,r}}(b)$ can be written as $g_b(x)\vp^{\mu_{0,r}}$.
So, in the same way as \cite[Proposition 6.12]{CI2}, we can show that $X_{w_{0,r}}(b)_{\sL_0}$ is equal to the set
$$\{gK\in X_{\lambda_{0,r}}(b)\mid \text{$g\sL_0\subset \sL_0$ and $v_L(\det(g))=C_{b,r}$}\},$$
and hence is closed in $X_{\lambda_{0,r}}(b)$.
By Lemma \ref{inj}, the closed subvarieties $hX_{w_{0,r}}(b)_{\sL_0}\ (h\in J/\JO)$ form a disjoint cover of $X_{\lambda_{0,r}}(b)$.
Since $X_{\lambda_{0,r}}(b)$ is locally of finite type, $hX_{w_{0,r}}(b)_{\sL_0}$ is open.
This result and Corollary \ref{ADLVIwahori} prove (i).

Next we prove (ii).
We use the notation of the proof of Corollary \ref{ADLVIwahori}.
By the proof of Theorem \ref{class}, we have
$$\SAdm(\lambda_{1,r})^\circ_{\cox}=\{c_j\vp^{c_j^{-1}(\lambda_{1,r})}\mid c_j=(1\ 2\ \cdots \ j\ n\ n-1\ \cdots\ j+1), 1\le j\le n-1\}.$$
It is easy to check that $c_j\vp^{c_j^{-1}(\lambda_{1,r})}\approx \vp^{\nu_{n-j}}\tau$ for all $1\le j\le n-1$.
Thus we have
$$X_{\lambda_{1,r}}(b)=\bigsqcup_{j=1}^{n-1}\pi(X_{\vp^{\nu_j}\tau}(b)).$$

In a similar way as (i), it follows that the map $X_{\vp^{\nu_j}\tau}(b)_0\rightarrow X_{\lambda_{1,r}}(b),\ gI\mapsto gK$ is an immersion.
So we regard each $hX_{\vp^{\nu_j}\tau}(b)_0(h\in J/\JO)$ as a locally closed subvariety of $X_{\lambda_{1,r}}(b)$.
Let $C_{b,r}^j=\frac{n(n-1)}{2}r+j+\sum_{j=1}^{n-1}\lfloor \frac{jk_0}{n_0}\rfloor$ ($1\le j\le n-1$).
Then, we can show that $X_{\vp^{\nu_{n-j_0}}\tau}(b)_0$ is equal to the set
$$\{gK\in X_{\lambda_{1,r}}(b)\mid \text{$g\sL_0\subset \sL_0$ and $v_L(\det(g))=C_{b,r}^{j_0}$}\}\setminus 
\bigsqcup_{j=1}^{j_0-1}X_{\vp^{\nu_{n-j}}}(b).$$
In particular, by the case $j_0=1$, $X_{\vp^{\nu_{n-1}}\tau}(b)_0$ is closed in $X_{\lambda_{1,r}}(b)$.
By Lemma \ref{inj} and Lemma \ref{closed}, the closed subvarieties $X_{\vp^{\nu_1}\tau}(b),\ldots, X_{\vp^{\nu_{n-2}}\tau}(b)$ and $hX_{\vp^{\nu_{n-1}}\tau}(b)_0\ (h\in J/\JO)$ form a disjoint cover of $X_{\lambda_{1,r}}(b)$.
Since $X_{\lambda_{1,r}}(b)$ is locally of finite type, $hX_{\vp^{\nu_{n-1}}\tau}(b)_0$ is open.
So $X_{\vp^{\nu_{n-1}}\tau}(b)$ is open and $X_{\vp^{\nu_{n-2}}\tau}(b)_0$ is closed.
Repeating the same argument, we see that each $hX_{\vp^{\nu_{n-j}}}(b)_0$ is closed and open in $X_{\lambda_{1,r}}(b)$.
This result and Corollary \ref{ADLVIwahori} prove (ii).

Finally, (iii) (resp.\ (iv)) follows from (i) (resp.\ (ii)) by Proposition \ref{dual}.
\end{proof}

\begin{rema}
For $X_{\lambda}(b)\neq \emptyset$, we have an explicit dimension formula:
$$\mathop{\mathrm{dim}} X_{\lambda}(b)=\langle\rho, \lambda-\nu_b \rangle-\frac{1}{2}\mathrm{def}(b),$$
where $\nu_b$ is the Newton vector of $b$, $\rho$ is half the sum of the positive roots, and $\mathrm{def}(b)$ is the defect of $b$.
For split groups, the formula was obtained in \cite{GHKR} and \cite{Viehmann}.
Using this formula, we can compute the dimension of $\A$ in Theorem \ref{maintheo}.
\end{rema}

\begin{rema}
In \cite{CV}, Chen and Viehmann define the $J$-stratification of affine Deligne-Lusztig varieties, which coincides with the Bruhat-Tits stratification of $X(\lambda, b)_P$ if $(G, \lambda, P)$ is of Coxeter type, see \cite{Gortz2}.
If $b$ is superbasic, it follows from \cite[Proposition 3.4]{CV} that the decomposition in Theorem \ref{maintheo} is also an example of the $J$-stratification.
\end{rema}

\bibliographystyle{myamsplain}
\bibliography{reference}

\providecommand{\bysame}{\leavevmode\hbox to3em{\hrulefill}\thinspace}
\providecommand{\MR}{\relax\ifhmode\unskip\space\fi MR }
\providecommand{\MRhref}[2]{%
  \href{http://www.ams.org/mathscinet-getitem?mr=#1}{#2}
}
\providecommand{\href}[2]{#2}
\begin{thebibliography}{10}

\bibitem{BS}
B.~Bhatt and P.~Scholze, \emph{Projectivity of the {W}itt vector affine
  {G}rassmannian}, Invent. Math. \textbf{209} (2017), no.~2, 329--423.
  \MR{3674218}

\bibitem{CI2}
C.~Chan and A.~Ivanov, \emph{Affine {D}eligne-{L}usztig varieties at infinite
  level}, Mathematische Annalen (2021).

\bibitem{CV}
M.~Chen and E.~Viehmann, \emph{Affine {D}eligne-{L}usztig varieties and the
  action of {$J$}}, J. Algebraic Geom. \textbf{27} (2018), no.~2, 273--304.
  \MR{3764277}

\bibitem{Gashi}
Q.~R. Gashi, \emph{On a conjecture of {K}ottwitz and {R}apoport}, Ann. Sci.
  \'{E}c. Norm. Sup\'{e}r. (4) \textbf{43} (2010), no.~6, 1017--1038.
  \MR{2778454}

\bibitem{Gortz2}
U.~G\"{o}rtz, \emph{Stratifications of affine {D}eligne-{L}usztig varieties},
  Trans. Amer. Math. Soc. \textbf{372} (2019), no.~7, 4675--4699. \MR{4009395}

\bibitem{GHKR}
U.~G\"{o}rtz, T.~J. Haines, R.~E. Kottwitz, and D.~C. Reuman, \emph{Dimensions
  of some affine {D}eligne-{L}usztig varieties}, Ann. Sci. \'{E}cole Norm. Sup.
  (4) \textbf{39} (2006), no.~3, 467--511. \MR{2265676}

\bibitem{GHKR2}
\bysame, \emph{Affine {D}eligne-{L}usztig varieties in affine flag varieties},
  Compos. Math. \textbf{146} (2010), no.~5, 1339--1382. \MR{2684303}

\bibitem{GH3}
U.~G\"{o}rtz and X.~He, \emph{Dimensions of affine {D}eligne-{L}usztig
  varieties in affine flag varieties}, Doc. Math. \textbf{15} (2010),
  1009--1028. \MR{2745691}

\bibitem{GH}
\bysame, \emph{Basic loci of {C}oxeter type in {S}himura varieties}, Camb. J.
  Math. \textbf{3} (2015), no.~3, 323--353. \MR{3393024}

\bibitem{GH2}
\bysame, \emph{Erratum to: {B}asic loci of {C}oxeter type in {S}himura
  varieties}, Camb. J. Math. \textbf{6} (2018), no.~1, 89--92. \MR{3786099}

\bibitem{GHN2}
U.~G\"{o}rtz, X.~He, and S.~Nie, \emph{{$\bf P$}-alcoves and nonemptiness of
  affine {D}eligne-{L}usztig varieties}, Ann. Sci. \'{E}c. Norm. Sup\'{e}r. (4)
  \textbf{48} (2015), no.~3, 647--665. \MR{3377055}

\bibitem{GHN}
\bysame, \emph{Fully {H}odge-{N}ewton decomposable {S}himura varieties}, Peking
  Math. J. \textbf{2} (2019), no.~2, 99--154. \MR{4060001}

\bibitem{GHN3}
\bysame, \emph{Basic loci of {C}oxeter type with arbitrary parahoric level},
  2020.

\bibitem{GHR}
U.~G\"{o}rtz, X.~He, and M.~Rapoport, \emph{Extremal cases of {R}apoport-{Z}ink
  spaces}, Journal of the Institute of Mathematics of Jussieu (2020), 1--56.

\bibitem{HN2}
X.~He and S.~Nie, \emph{Minimal length elements of extended affine {W}eyl
  groups}, Compos. Math. \textbf{150} (2014), no.~11, 1903--1927. \MR{3279261}

\bibitem{HZZ}
X.~He, R.~Zhou, and Y.~Zhu, \emph{Stabilizers of irreducible components of
  affine {D}eligne-{L}usztig varieties}, 2021.

\bibitem{Macdonald}
I.~G. Macdonald, \emph{Affine {H}ecke algebras and orthogonal polynomials},
  Cambridge Tracts in Mathematics, vol. 157, Cambridge University Press,
  Cambridge, 2003. \MR{1976581}

\bibitem{NP}
B.~C. Ng\^{o} and P.~Polo, \emph{R\'{e}solutions de {D}emazure affines et
  formule de {C}asselman-{S}halika g\'{e}om\'{e}trique}, J. Algebraic Geom.
  \textbf{10} (2001), no.~3, 515--547. \MR{1832331}

\bibitem{PR}
G.~Pappas and M.~Rapoport, \emph{Twisted loop groups and their affine flag
  varieties}, Adv. Math. \textbf{219} (2008), no.~1, 118--198, With an appendix
  by T. Haines and Rapoport. \MR{2435422}

\bibitem{RR}
M.~Rapoport and M.~Richartz, \emph{On the classification and specialization of
  {$F$}-isocrystals with additional structure}, Compositio Math. \textbf{103}
  (1996), no.~2, 153--181. \MR{1411570}

\bibitem{Shimada}
R.~Shimada, \emph{Geometric structure of affine {D}eligne-{L}usztig varieties
  for {${\rm GL}_3$}}, 2021.

\bibitem{Viehmann}
E.~Viehmann, \emph{The dimension of some affine {D}eligne-{L}usztig varieties},
  Ann. Sci. \'{E}cole Norm. Sup. (4) \textbf{39} (2006), no.~3, 513--526.
  \MR{2265677}

\end{thebibliography}
\end{document}